\documentclass[reqno,12pt,a4paper]{amsart}

\usepackage{tikz}

\usepackage{bbm}
\usepackage{amscd}
\usepackage{mathrsfs}
\usepackage{bbm}\usepackage{graphicx}
\usepackage{mathrsfs}
\usepackage{amsmath}
\usepackage{amsthm}
\usepackage{amsfonts}
\usepackage{amssymb}
\usepackage{latexsym}
\usepackage[all]{xy}
\usepackage[colorlinks=true]{hyperref}

\newtheorem{theorem}{Theorem}[section]
\newtheorem{proposition}[theorem]{Proposition}

\newtheorem{definition}[theorem]{Definition}
\newtheorem{lemma}[theorem]{Lemma}

\newtheorem*{ack*}{ACKNOWLEDGEMENTS}

\newcommand{\opname}[1]{\operatorname{\mathsf{#1}}}

\renewcommand{\mod}{\opname{mod}\nolimits}

\newcommand{\proj}{\opname{proj}\nolimits}

\newcommand{\Mod}{\opname{Mod}\nolimits}

\newcommand{\Hom}{\opname{Hom}}
\newcommand{\End}{\opname{End}}

\numberwithin{equation}{section}

\begin{document}

\title[]{The ST correspondence for proper non-positive dg algebras}

\author{Houjun Zhang}
\address{Houjun Zhang, Department of Mathematics, Nanjing University, Nanjing 210093, P. R. China
}
\email{zhanghoujun@nju.edu.cn}

\begin{abstract} Let $A$ be a proper non-positive dg algebra over a field $k$. For a simple-minded collection of the finite-dimensional derived category $\mathcal{D}_{fd}(A)$, we construct a 'dual' silting object of the perfect derived category $\mathrm{per}(A)$ by using the Koszul duality for dg algebras. This induces a one-to-one correspondence between the equivalence classes of silting objects in $\mathrm{per}(A)$ and algebraic t-structures of $\mathcal{D}_{fd}(A)$.
\end{abstract}

\maketitle

\section{Introduction}\label{s:introduction}
\medskip
Let $A$ be a proper non-positive dg algebra over a field $k$. In this paper we establish the following correspondence: there is a one-to-one correspondence between

$(1)$ equivalence classes of silting objects in the perfect derived category $\mathrm{per}(A)$,

$(2)$ algebraic $t$-structures of the finite-dimensional derived category $\mathcal{D}_{fd}(A)$,\\
which we will call the ST correspondence(S=silting objects, T=t-structures). Such a correspondence was first established by Keller and Vossieck \cite{KellerVossieck88} for path algebras of Dynkin quivers and in recent years it was studied in the following settings: homologically smooth non-positive dg algebras \cite{KellerNicolas11}, finite-dimensional algebras \cite{KoenigYang14,BuanReitenThomas12,RickardRouquier17} and proper non-positive dg algebras over an algebraically closed field $k$ \cite{SuYang19}.

Since algebraic t-structures are generated by simple-minded collections \cite{Al-Nofayee09,KoenigYang14,Schnurer13}, the key point to establish the above  ST correspondence is to construct a silting object of $\mathrm{per}(A)$ from a simple-minded collection of $\mathcal{D}_{fd}(A)$. There are two approaches to achieve this. The first approach was provided by Rickard. Let $A$ be a finite-dimensional symmetric algebra. Given a simple-minded collection of $\mathcal{D}^{b}(\mod A)(=\mathcal{D}_{fd}(A))$, he \cite{Rickard02} constructed a sequence of morphisms in the unbounded derived category $\mathcal{D}(\Mod A)$ of $A$-modules and showed that the homotopy colimit of this sequence is a tilting object of $K^{b}(\proj A)(=\mathrm{per}(A))$. This method was later shown to be applicable to homologically smooth non-positive dg algebras with finite-dimensional zeroth cohomology \cite{KellerNicolas11} and finite-dimensional algebras \cite{KoenigYang14}, but instead of a tilting object one obtains a silting object. The second approach was provided by Su and Yang \cite{SuYang19} and it uses Koszul duality between dg algebras and $A_{\infty}$-algebras.

In this paper, we follow Su and Yang's approach, but on the Koszul dual side, we use dg algebras instead of $A_{\infty}$-algebras. Our main result is:

\begin{theorem} Let $A$ be a proper non-positive dg algebra over a field $k$ and let $\{X_{1},\ldots,X_{n}\}$ be a simple-minded collection of $\mathcal{D}_{fd}(A)$ with endomorphism algebras $R_{1},\ldots,R_{n}$. Then there exists a unique (up to isomorphism) silting object $M = M_{1} \oplus\ldots\oplus M_{n}$ of $\mathrm{per}(A)$ such that for $1 \leq i, j \leq n$ and $p \in Z$
\[\Hom _{\mathcal{D}_{fd}(A)}(M_{i}, \Sigma^{p}X_{j})=\left\{\begin{array}{ll}

_{R_{j}}R_{j}&\mathrm{if}\ i=j\ \mathrm{and}\ p=0,\\

0&\mathrm{otherwise}.

\end{array}\right.\]
\end{theorem}
One main reason for Su and Yang to use $A_{\infty}$-algebras is that the simple modules over the $A_{\infty}$-Koszul dual are easily constructed and characterised. In this paper, we use the results of Keller and Nicol\'{a}s \cite{KellerNicolas13} to construct and characterise the simple modules over the dg-Koszul dual. In this way, the constraint in \cite[Theorem 1.1]{SuYang19} that the given simple-minded collection is elementary is removed.

\noindent {\bf Acknowledgements} The author would like to thank Dong Yang for introducing me this topic and for his consistent encouragement and support. He also thanks Zongzhen Xie for
her careful reading of the paper and for her helpful comments. The project is funded by China Postdoctoral Science Foundation (2020M681540).

\section{Preliminaries}

In this section, we recall some notions and results about silting objects, simple-minded collections and dg algebras. Let $\mathcal{C}$ be a triangulated category with shift functor $\Sigma$. Throughout this paper, let $k$ be a field and let $D=\mathrm{Hom}_{k}(?,k)$ be the $k$-dual.

\subsection{Silting objects.}
For a subcategory or a set of objects $S$ of $\mathcal{C}$, we denote by $\mathrm{thick}(S)$ the $\emph thick~~subcategory$ of $\mathcal{C}$ generated by $S$, i.e., the smallest triangulated subcategory of $\mathcal{C}$ which contains $S$ and which is closed under isomorphisms and direct summands. In the following, we give the definition of silting objects. For more details we refer to \cite{AiharaIyama12} and \cite{KellerVossieck88}.

\begin{definition} Let $M$ be an object in $\mathcal{C}$.

$(1)$ $M$ is said to be \emph{presilting} if $\mathrm{Hom}_{\mathcal{C}}(M,\Sigma^{n}M)=0$ for all $n>0$.

$(2)$ $M$ is said to be \emph{silting} if it is presilting and $\mathcal{C}=\mathrm{thick}(M)$.
\end{definition}

\noindent Two silting objects $M$ and $M^{\prime}$ of $\mathcal{C}$ are said to be $\emph equivalent$ if $\mathrm{add}(M)=\mathrm{add}(M^{\prime})$.

\subsection{Simple-minded collections.}

\begin{definition} A collection $X_{1},\ldots,X_{n}$ of objects of $\mathcal{C}$ is said to be \emph{simple-minded} if the following conditions hold for $1\leq i,j\leq n$

$(1)$ $\mathrm{Hom}_{\mathcal{C}}(X_{i},\Sigma^{m}X_{j})=0$ for all $m<0$,

$(2)$ $\mathrm{End}(X_{i})$ is a division algebra and $\mathrm{Hom}_{\mathcal{C}}(X_{i},X_{j})$ vanishes for $i\neq j$,

$(3)$ $X_{1},\ldots,X_{n}$ generate $\mathcal{C}$, i.e., $\mathcal{C}=\mathrm{thick}(X_{1},\ldots,X_{n})$.
\end{definition}
For example, for a finite-dimensional algebra $\Lambda$, a complete collection of pairwise non-isomorphic simple modules is a simple-minded collection in $\mathcal{D}^{b}(\mod \Lambda)$.
\subsection{Non-positive dg algebras.}
Let $A$ be a dg $k$-algebra and let $M$ and $N$ be two dg $A$-modules. The \emph{morphism~complex} of $M$ and $N$ is $\mathcal{H}om_{A}(M,N)$, whose degree $n$ component consists of those $A$-linear maps from $M$ to $N$ which are homogeneous of degree $n$. The differential of $\mathcal{H}om_{A}(M, N)$ is given by
\begin{center}
$ d(f) = d_{N}\circ f- (-1)^{|f|}f \circ d_{M}$,
\end{center}
where $f \in \mathcal{H}om_{A}(M, N)$ is homogeneous of degree $|f|$. Let $\mathcal{C}_{dg}(A)$ be the dg category of dg $A$-modules and let $K(A)$ be the homotopy category. Then
$$
\Hom_{\mathcal{C}_{dg}(A)}(M,N)=\mathcal{H}om_{A}(M,N), \mathrm{Hom}_{K(A)}(M,N)=H^{0}\mathcal{H}om_{A}(M,N).
$$
The derived category $\mathcal{D}(A)$ of dg $A$-modules is defined as the triangle quotient of $K(A)$ by acyclic dg $A$-modules, see \cite{Keller94,Keller06d}. Let $\mathrm{per}(A) = \mathrm{thick}_{\mathcal{D}(A)}(A_{A})$ be the
thick subcategory of $\mathcal{D}(A)$ generated by $A_{A}$ and $\mathcal{D}_{fd}(A)$ the full subcategory of $\mathcal{D}(A)$ consisting of those dg $A$-modules with finite-dimensional total cohomology. In the case $A$ is a $k$-algebra, we have $\mathcal{D}(A)=\mathcal{D}(\Mod A)$, $\mathrm{per}(A)\simeq K^{b}(\proj A)$. If $A$ is a finite-dimensional algebra over the field $k$, then $\mathcal{D}_{fd}(A)\simeq\mathcal{D}^{b}(\mod A)$.


A dg $A$-module $M$ is said to be \emph{K}-\emph{projective} if $\mathcal{H}om_{A}(M, N)$ is acyclic for any acyclic dg $A$-module $N$. If $M$ is $K$-projective, then there
is a canonical isomorphism
\begin{center}
$\mathrm{Hom}_{\mathcal{D}(A)}(M,\Sigma^{p}N)\cong\mathrm{Hom}_{K(A)}(M,\Sigma^{p}N)$.
\end{center}
In particular, let $\mathcal{E}nd_{A}(M)=\mathcal{H}om_{A}(M,M)$ denote the dg endomorphism algebra of $M$; then
\begin{center}
$\mathrm{Hom}_{\mathcal{D}(A)}(M,\Sigma^{p}M)=H^{p}\mathcal{E}nd_{A}(M)$.
\end{center}
For a dg $A$-module $M$, a $K$-projective resolution of $M$ is a quasi-isomorphism $\mathbf{p}M\rightarrow M$ of dg $A$-modules with $\mathbf{p}M$ being $K$-projective. By \cite[Theorem 3.1]{Keller94}, $K$-projective resolutions always exist.


A dg $k$-algebra $A$ is said to be $\emph non$-$\emph positive$ if the degree $i$ component $A^{i}$ vanishes for all $i>0$ and $\emph proper$ if $A$ has finite-dimensional total cohomology. Let $A$ be a proper non-positive dg $k$-algebra. Then both $\mathcal{D}_{fd}(A)$ and $\mathrm{per}(A)$ are Krull-Schmidt. Moreover, $\mathrm{per}(A)\subseteq \mathcal{D}_{fd}(A)$ and $\mathrm{thick}(D(_{A}A))\subseteq \mathcal{D}_{fd}(A)$. There is a triangle functor $\nu: \mathcal{D}(A)\longrightarrow \mathcal{D}(A)$ which restricts to a triangle equivalence $\nu: \mathrm{per}(A)\longrightarrow \mathrm{thick}(D(_{A}A))$. We have the Auslander-Reiten formula
\begin{center}
$D\mathrm{Hom}(M,N)\cong \mathrm{Hom}(N,\nu(M))$
\end{center}
for $M\in \mathrm{per}(A)$ and $N\in \mathcal{D}(A)$. See \cite[Section 10]{Keller94}.

Let $S_{1},\ldots,S_{n}$ be a complete set of pairwise non-isomorphic simple $H^{0}(A)$-modules and view them as dg $A$-modules via the homomorphism $A\longrightarrow H^{0}(A)$. Recall that $\{S_{1},\ldots,S_{n}\}$ is a simple-minded collection in $\mathcal{D}_{fd}(A)$, see \cite[Theorem A.1 (c)]{BrustleYang13}. Let $U_{1},\ldots,U_{n}$ be their endomorphism algebras.

Since $\End_{\mathcal{D}(A)}(A)=H^{0}(A)$, the functor $H^{0}=\Hom_{\mathcal{D}(A)}(A,?)$ restricts to an equivalence $\mathrm{add}_{\mathcal{D}(A)}(A)\longrightarrow \proj H^{0}(A)$. Therefore there are indecomposable objects $P_{1},\ldots,P_{n}\in \mathrm{add}_{\mathcal{D}(A)}(A)\subseteq \mathrm{per}(A)$ such that $H^{0}(P_{1}),\ldots,H^{0}(P_{n})$ are projective covers of $S_{1},\ldots,S_{n}$, respectively. Thus, we obtain that
\[\Hom _{\mathcal{D}(A)}(P_{i}, \Sigma^{p}S_{j})=\left\{\begin{array}{ll}

_{U_{j}}U_{j}&\mathrm{if}\ i=j\ \mathrm{and}\ p=0,\\

0&\mathrm{otherwise}.

\end{array}\right.\]
Moreover, the collection $\{S_{1},\ldots,S_{n}\}$ is determined by this property. Namely, fix $1\leq j\leq n$ and let $M\in \mathcal{D}(A)$ be such that

\[\Hom _{\mathcal{D}(A)}(P_{i}, \Sigma^{p}M)=\left\{\begin{array}{ll}

_{U_{j}}U_{j}&\mathrm{if}\ i=j\ \mathrm{and}\ p=0,\\

0&\mathrm{otherwise},

\end{array}\right.\]
then $M\cong S_{j}$ in $\mathcal{D}(A)$. See \cite[Section 5.1]{SuYang19}.

Let $I_{i}=\nu(P_{i})$ for $1\leq i\leq n$. Then by the Auslander-Reiten formula we have

\[\Hom _{\mathcal{D}(A)}(S_{j}, \Sigma^{p}I_{i})=\left\{\begin{array}{ll}

(U_{j})_{U_{j}}&\mathrm{if}\ i=j\ \mathrm{and}\ p=0,\\

0&\mathrm{otherwise}.

\end{array}\right. \eqno{(\ast)}\]

\subsection{Cohomologically strictly positive dg algebras.}

A dg $k$-algebra is said to be $\emph cohomologically$ $\emph strictly$ $\emph positive$ if $H^{p}(A)=0$ for all $p<0$ and $H^{0}(A)$ is a semi-simple $k$-algebra.

Let $A$ be a cohomologically strictly positive dg $k$-algebra. Since $H^{0}(A)$ is semi-simple, by \cite[Lemma 4.5]{KellerNicolas13}, there exists a decomposition into indecomposables $A=\mathop{\bigoplus}\limits_{i=1}^{n}e_{i}A$ of $A$ in $\mathcal{D}(A)$ such that $H^{0}A=\mathop{\bigoplus}\limits_{i=1}^{n}H^{0}(e_{i}A)$ is a decomposition into simples of $H^{0}A$ in $\Mod H^{0}A$, where $e_{1},\ldots,e_{n}$ is a complete set of primitive orthogonal idempotents of $A$. We view $H^{0}(e_{i}A)$ as a graded $H^{\ast}(A)$-module via the graded algebra homomorphism $H^{\ast}(A)\longrightarrow H^{0}(A)$. Then by \cite[Corollary 4.7]{KellerNicolas13}, there is a unique (up to isomorphism) dg $A$-module $S_{i}$
such that the graded $H^{\ast}(A)$-module $H^{\ast}(S_{i})$ is isomorphic to $H^{0}(e_{i}A)$ for any $1\leq i\leq n$. We call $S_{1},\ldots,S_{n}$ the simple $A$-modules.

\begin{lemma}\cite[Lemma 3.12]{AdachiMizunoYang19} $\mathop{\bigoplus}\limits_{i=1}^{n}S_{i}$ is a silting object in $\mathcal{D}_{fd}(A)$.
\end{lemma}

The following lemma is very important.

\begin{lemma}
Fix $1\leq i\leq n$. If for any $1\leq j\leq n$ and $p\in \mathbb{Z}$, we have
\[\Hom_{\mathcal{D}(A)}(e_{j}A,\Sigma^{p}M)=\left\{\begin{array}{ll}
(T_{i})_{T_{i}}&\mathrm{if}\ j=i\ \mathrm{and}\ p=0,\\

0&\mathrm{otherwise},

\end{array}\right.\]
where $T_{i}$ is the endomorphism algebra of $e_{i}A$.
Then $M$ is isomorphic to $S_{i}$ in $\mathcal{D}(A)$.
\end{lemma}

\begin{proof}
By assumption, we have
\[H^{p}(Me_{j})=\left\{\begin{array}{ll}
(T_{i})_{T_{i}}&\mathrm{if}\ j=i\ \mathrm{and}\ p=0,\\

0&\mathrm{otherwise}.
\end{array}\right.\]
It follows that $H^{\ast}(M)$ is a simple module over $H^{\ast}(A)$ with cohomologies concentrated in degree 0.
Then, by the uniqueness result, $M$ is isomorphic to $S_{i}$ in $\mathcal{D}(A)$.
\end{proof}



\section{Proof of Theorem 1.1}

In this section, we give a proof of Theorem 1.1. Precisely, let $A$ be a proper non-positive dg $k$-algebra, for a simple-minded collection $X_{1},\ldots,X_{n}$ of $\mathcal{D}_{fd}(A)$, we will construct a silting object which is 'dual' to $X_{1},\ldots,X_{n}$.

Recall that the simple $H^{0}(A)$-modules $S_{1},\ldots,S_{n}$ is a simple-minded collection of $\mathcal{D}_{fd}(A)$. Let $A^{\ast}=\mathcal{E}nd_{A}(\mathop{\bigoplus}\limits_{j=1}^{n}\mathbf{p}S_{j})$. Then we have the following
result.

\begin{lemma}
$A^{\ast}$ is a cohomologically strictly positive dg algebra and there is a triangle equivalence $\Phi: \mathcal{D}_{fd}(A)\longrightarrow \mathrm{per}(A^{\ast})$ which takes $S_{j}$ to $e_{j}A^{\ast}$ for all $1\leq j\leq n$ and the following formula holds for all $1\leq i,j\leq n$ and $p\in \mathbb{Z}$
\[\Hom _{\mathcal{D}(A^{\ast})}(e_{j}A^{\ast}, \Sigma^{p}\Phi(I_{i}))=\left\{\begin{array}{ll}

(U_{j})_{U_{j}}&\mathrm{if}\ i=j\ \mathrm{and}\ p=0,\\

0&\mathrm{otherwise}.

\end{array}\right. \eqno{(\ast\ast)}\]
\end{lemma}

\begin{proof}
By \cite[Lemma 6.1]{Keller94}, $\Phi=?\mathop{\otimes}\limits_{A^{\ast}}^{\mathbb{L}}(\mathop{\bigoplus}\limits_{j=1}^{n}\mathbf{p}S_{j}): \mathcal{D}_{fd}(A)\longrightarrow \mathrm{per}(A^{\ast})$ is a triangle equivalence which takes $S_{j}$ to $e_{j}A^{\ast}$ for all $1\leq j\leq n$. Thus, there is an isomorphism
\begin{center}
$H^{p}(A^{\ast})=\Hom_{\mathrm{per}(A^{\ast})}(A^{\ast},\Sigma^{p}A^{\ast})\cong\Hom_{\mathcal{D}_{fd}(A)}(\mathop{\bigoplus}\limits_{j=1}^{n}S_{j},\Sigma^{p}\mathop{\bigoplus}\limits_{j=1}^{n}S_{j})$
\end{center}
for all $p\in \mathbb{Z}$. Since $S_{1},\ldots,S_{n}$ is a simple-minded collection of $\mathcal{D}_{fd}(A)$, it follows that $A^{\ast}$ is a cohomologically strictly positive dg algebra. The formula $(\ast\ast)$ follows from the formula $(\ast)$.
\end{proof}

By Lemma 2.4, the formula $(\ast\ast)$ implies that $\Phi(I_{1}),\ldots,\Phi(I_{n})$ are isomorphic to the simple modules over $A^{\ast}$. Thus, by Lemma 2.3, we have
\begin{center}
$\mathrm{thick}_{\mathcal{D}(A^{\ast})}(\Phi(I_{1}),\ldots,\Phi(I_{n}))=D_{fd}(A^{\ast})$.
\end{center}
It follows that the equivalence $\Phi$ restricts to a triangle equivalence
\begin{center}
$\mathrm{thick}_{\mathcal{D}(A)}(D(_{A}A))\rightarrow D_{fd}(A^{\ast})$.
\end{center}
Moreover, we obtain that $\Phi(X_{1}),\ldots,\Phi(X_{n})$ is a simple-minded collection in $\mathrm{per}(A^{\ast})$. By \cite[Lemma 6.1]{Keller94}, there is a triangle equivalence $\Psi=?\mathop{\otimes}\limits_{\mathcal{X}}^{\mathbb{L}}(\mathop{\bigoplus}\limits_{i=1}^{n}\mathbf{p}\Phi(X_{i})): \mathcal{D}(A^{\ast})\longrightarrow \mathcal{D}(\mathcal{X})$ taking $\Phi(X_{i})$ to $e_{i}\mathcal{X}$ for all $1\leq i\leq n$, where $\mathcal{X}=\mathcal{E}nd_{A}(\mathop{\bigoplus}\limits_{i=1}^{n}\mathbf{p}\Phi(X_{i}))$.  By \cite[Lemma 3.1]{KalckYang18a}, it restricts to triangle equivalences $\mathrm{per}(A^{\ast})\longrightarrow \mathrm{per}(\mathcal{X})$ and $\mathcal{D}_{fd}(A^{\ast})\longrightarrow \mathcal{D}_{fd}(\mathcal{X})$, and we obtain the following commutative diagram
$$\centerline{\xymatrix{&\mathcal{D}_{fd}(A)\ar[r]^{\Phi}&\mathrm{per}(A^{\ast})\ar[r]^{\Psi}&\mathrm{per}(\mathcal{X})\\
\mathrm{per}(A)\ar[r]^(.4){\nu}&\mathrm{thick}(D(_{A}A))\ar@{^{(}->}[u]\ar[r]^(.6){\Phi}&\mathrm{D}_{fd}(A^{\ast})\ar@{^{(}->}[u]\ar[r]^{\Psi}&\mathrm{D}_{fd}(\mathcal{X})\ar@{^{(}->}[u]}}$$

Assume that $R_{1},\ldots,R_{n}$ are the endomorphism algebras of the simple-minded collection $X_{1},\ldots,X_{n}$. Let $W_{1},\ldots,W_{n}$ be the simple modules over $\mathcal{X}$ and let $T_{1},\ldots,T_{n}$ be their images under a quasi-inverse of the equivalence $\Psi\circ\Phi$.

\begin{proposition}
$(1)$ $T=\mathop{\bigoplus}\limits_{i=1}^{n}T_{i}$ is a silting object of $\mathrm{thick}(D(_{A}A))$.

$(2)$ For $1\leq i,j\leq n$, and $p=\mathbb{Z}$,
\[\Hom _{\mathcal{D}_{fd}(A)}(X_{j}, \Sigma^{p}T_{i})=\left\{\begin{array}{ll}

(R_{j})_{R_{j}}&\mathrm{if}\ i=j\ \mathrm{and}\ p=0,\\

0&\mathrm{otherwise}.

\end{array}\right.\]

$(3)$ $\nu^{-1}T$ is a silting object of $\mathrm{per}(A)$.

$(4)$ For $1\leq i,j\leq n$, and $p=\mathbb{Z}$,
\[\Hom _{\mathcal{D}_{fd}(A)}(\nu^{-1}T_{i}, \Sigma^{p}X_{j})=\left\{\begin{array}{ll}

_{R_{j}}R_{j}&\mathrm{if}\ i=j\ \mathrm{and}\ p=0,\\

0&\mathrm{otherwise}.

\end{array}\right.\]
\end{proposition}

\begin{proof}
(1) By Lemma $2.3$, $W=\mathop{\bigoplus}\limits_{i=1}^{n}W_{i}$ is a silting object of $\mathcal{D}_{fd}(\mathcal{X})$. Then $T$ is a silting object of $\mathrm{thick}(D(_{A}A))$ because $\Psi\circ\Phi$ is a triangle equivalence.

(2) The triangle equivalence $\Psi\circ\Phi$ sends $X_{i}$ to $e_{i}\mathcal{X}$ for all $1\leq i\leq n$. This shows that the endomorphism algebras of $e_{i}\mathcal{X}$ are $R_{i}$. Then for $1\leq i,j\leq n$ and $p\in\mathbb{Z}$, we obtain that
\[\Hom _{\mathcal{D}_{fd}(\mathcal{X})}(e_{j}\mathcal{X}, \Sigma^{p}W_{i})=\left\{\begin{array}{ll}

(R_{j})_{R_{j}}&\mathrm{if}\ i=j\ \mathrm{and}\ p=0,\\

0&\mathrm{otherwise}.

\end{array}\right.\]

(3) This follows from (1) because $\nu: \mathrm{per}(A)\longrightarrow \mathrm{thick}(D(_{A}A))$ is a triangle equivalence.

(4) This follows by Auslander-Reiten formula and (2).
\end{proof}


\begin{proof}[Proof of Theorem 1.1] Take $M=\nu^{-1}T$. By Proposition 3.2 (3) and (4), it remains to show the uniqueness of $M$.






Let $N=\mathop{\bigoplus}\limits_{i=1}^{n}N_{i}$ be an object of $\mathrm{per}(A)$ such that for $1 \leq i, j \leq r$ and $p \in \mathbb{Z}$
\[\Hom _{\mathcal{D}_{fd}(A)}(N_{i}, \Sigma^{p}X_{j})=\left\{\begin{array}{ll}

_{R_{j}}R_{j}&\mathrm{if}\ i=j\ \mathrm{and}\ p=0,\\

0&\mathrm{otherwise}.

\end{array}\right.\]
Then by the Auslander-Reiten formula we have
\[\Hom _{\mathcal{D}_{fd}(A)}(X_{j},\Sigma^{p}\nu N_{i})=\left\{\begin{array}{ll}

(R_{j})_{R_{j}}&\mathrm{if}\ i=j\ \mathrm{and}\ p=0,\\

0&\mathrm{otherwise}.

\end{array}\right.\]
Applying the triangle equivalence $\Psi\circ\Phi$ we obtain
\[\Hom _{\mathcal{D}_{fd}(\mathcal{X})}(e_{j}\mathcal{X},\Sigma^{p}\Psi\circ\Phi\circ\nu N_{i})=\left\{\begin{array}{ll}

(R_{j})_{R_{j}}&\mathrm{if}\ i=j\ \mathrm{and}\ p=0,\\

0&\mathrm{otherwise}.

\end{array}\right.\]
Therefore, by Lemma 2.4, we have $\Psi\circ\Phi\circ\nu N_{i}\cong W_{i}$. Thus $N_{i}\cong M_{i}$.
\end{proof}

As a consequence, we obtain the following result. It is similar to \cite[Theorem 5.5]{SuYang19}, so we omit the proof.

\begin{theorem}
Let $A$ be a proper non-positive dg algebra over the field $k$. Then their are one-to-one correspondences which commute with mutation and which preserve partial orders between

(1) equivalence classes of silting objects in $\mathrm{per}(A)$,

(2) isomorphism classes of simple-minded collections in $\mathcal{D}_{fd}(A)$,

(3) algebraic t-structures of $\mathcal{D}_{fd}(A)$, i.e. bounded t-structures of $\mathcal{D}_{fd}(A)$ with length heart,

(4) bounded co-t-structures of $\mathrm{per}(A)$.
\end{theorem}


\def\cprime{$'$}
\providecommand{\bysame}{\leavevmode\hbox to3em{\hrulefill}\thinspace}
\providecommand{\MR}{\relax\ifhmode\unskip\space\fi MR }
\providecommand{\MRhref}[2]{%
  \href{http://www.ams.org/mathscinet-getitem?mr=#1}{#2}
}
\providecommand{\href}[2]{#2}

\end{document}